\documentclass[12pt]{article}
\usepackage{geometry}
\usepackage{mathtools}
\usepackage{amsmath}
\usepackage{amsthm}
\usepackage{amssymb}
\usepackage{xcolor}

\makeatletter
\@addtoreset{equation}{section}
\makeatother

\newtheorem{theorem}{Theorem}
\newtheorem{property}{Property}
\newtheorem{proposition}{Proposition}
\newtheorem{lemma}{Lemma}

\title{On $L$-functions of Hecke characters and \\ anticyclotomic towers}
\date{}

\begin{document}
	\maketitle
	{\centering Haijun Jia \par }
	\begin{abstract}
		\small
		In this paper, we generalize a work of Rohrlich. Let $K/\mathbb{Q}$ be an imaginary quadratic field and $\phi$ be a Hecke character of $K$ of infinite type (1,0) whose restriction to $\mathbb{Q}$ is the quadratic character corresponding to $K/\mathbb{Q}$. We consider a class of Hecke characters $\chi$, which are anticyclotomic twists of $\phi$ with ramification in a prescribed finite set of primes. We shall prove the central vanishing order of the Hecke $L$-function $L(s,\chi$) attached to each $\chi$ is 0 or 1 depending on the root number $W(\chi)$ for all but finitely many such $\chi$. 
	\end{abstract}
	
	\section{Introduction}
	Let $K$ be an imaginary quadratic field with class number $h$ and $\mathcal{O}$ be the ring of integers of $K$. Let $M$ be an abelian extension of $K$ which can be infinte. We say $M/K$ is anticyclotomic if the nontrivial element of Gal($K/\mathbb{Q}$) acts on Gal($M/K$) by inversion. Let $P$ be a fixed finite set of rational primes. Let $L$ be the compositum of all anticyclotomic extensions of $K$ which are unramified outside $P$. The field $L$ can be also described as the union of all ring class fields of $K$ with conductor divisible only by primes in $P$. Using class field theory, we know that Gal($L/K$) is isomorphic to the product of a finite group and  
	$$
	\prod_{p\in P}\mathbb{Z}_p,
	$$
	where $\mathbb{Z}_p$ is the ring of $p$-adic integers.
	
	Let $\phi$ be a Hecke character (also called Gr\"ossen character in many literatures) of $K$ with infinite type (1,0), and $\mathfrak{f}(\phi)$ be the conductor of $\phi$. We say $\phi$ is equivariant with respect to complex conjugation (or equivariant for simplicity), if for all integral ideal $\mathfrak{a}$ of $K$, we have 
	$$
	\phi(\bar{\mathfrak{a}})=\overline{\phi(\mathfrak{a})}.
	$$
	We assume that $\phi$ is equivariant with respect to complex conjugation in the following discussions.
	
	For a given finite order character $\rho:\mbox{Gal}(L/K)\rightarrow\mathbb{C}^{\times}$, we view it as an idele class character through
	$$
	\mathbb{A}_{K}^{\times}/K^{\times}\simeq {\rm Gal}(K^{ab}/K)\twoheadrightarrow {\rm Gal}(L/K)\rightarrow\mathbb{C}^{\times},
	$$
	and let $\phi\rho$ denote the primitive Hecke character given by their product. Let
	$$
	X=\{\chi\,|\ \chi=\phi\rho\thickspace \mbox{for some finite order character $\rho$ of Gal$(L/K)$}\}.
	$$
	
	We note that $\rho$ is equivariant because $L/K$ is anticyclotomic, so $\chi$ is also equivariant for all $\chi\in X$. Hence
	$$
	L(s,\chi)=\sum_{\mathfrak{a}}\chi(\mathfrak{a})N\mathfrak{a}^{-s}	=\sum_{\mathfrak{a}}\chi(\bar{\mathfrak{a}})N\mathfrak{a}^{-s}
	=\sum_{\mathfrak{a}}\overline{\chi}(\mathfrak{a})N\mathfrak{a}^{-s}
	=L(s,\bar{\chi}).
	$$
	For the real number $s$,
	$$L(s,\bar{\chi})=\lim_{t\rightarrow\infty}\sum_{N\mathfrak{a}\leq t}\bar{\chi}(\mathfrak{a})N\mathfrak{a}^{-s}
	=\lim_{t\rightarrow\infty}\overline{\sum_{N\mathfrak{a}\leq t}\chi(\mathfrak{a})N\mathfrak{a}^{-s}}
	=\overline{L(s,\chi)},
	$$
	so $L(s,\chi)$ is real. Therefore, in the functional equation
	$$
	\Lambda(s,\chi)=W(\chi)\Lambda(2-s,\bar{\chi})=W(\chi)\Lambda(2-s,\chi),
	$$
	the root number $W(\chi)=1$ or $-1$, and $W(\chi)$ determines the parity of the vanishing order of $L(s,\chi)$ at $s=1$. Here $\Lambda$ is the completed $L$-function which is defined exactly in the following section. However, $W(\chi)$ in fact almost determines the vanishing order, and we give the main result of this paper:
	
	\begin{theorem}
		For all but finitely many $\chi\in X$,
		
		${\rm ord}_{s=1}L(s,\chi)=\left\{
		\begin{aligned}
			0,\qquad &\mbox{if }\ W(\chi)=1,\\
			1,\qquad &\mbox{if }\ W(\chi)=-1.
		\end{aligned}\right.$
	\end{theorem}
	
	Our result is a generalization of Rohrlich $\cite{Roh84a}$, and now we compare it with Rohrlich $\cite{Roh84a}$. Rohrlich has actually proved the case when class number $h$ of $K$ is 1. As there are only 9 imaginary quadratic fields with class number 1, it is more satisfying to remove the class number 1 assumption. We explain this from an arithmetic point of view. We know that there is a cusp form of weight 2, which has trivial central character in our case, associated with $\chi$ (see $\cite{Hecke}$, $\cite{Shi71}$, $\cite{Shi72}$). Then there is an abelian variety $A$ associated to this cusp form (see $\cite{DarDiaTay}$, Definition 1.44). This abelian variety has to be defined over $\mathbb{Q}$ and of GL$_2$-type, and has complex multiplication over $K$ as well (see $\cite{YuanZhangZhang}$, Chapter 3.2 for definition of GL$_2$-type, and see $\cite{Milne}$ for definition of abelian variety with complex multiplication). Every elliptic curve defined over $\mathbb{Q}$ with complex multiplication by the ring of integers $\mathcal{O}$ of $K$ can be obtained through this approach. This is because for each elliptic curve of this type, there is Hecke character $\phi$ of $K$ associated to $E$, i.e. satisfying 	
	$$
	L(s,E/\mathbb{Q})=L(s,\phi),
	$$ 
	by Deuring's result (see $\cite{Silverman}$, Chapter II, Theorem 10.5). This Hecke character is equivariant because it is the Hecke character determined by an elliptic curve over $\mathbb{Q}$ (see $\cite{ShiGor}$, page 519, formula (4.6)), and also has infinite type (1,0), so we use the same notation $\phi$ when we state our main result. $E$ is given by this $\phi$ under the corresponding principle. Note that the class number $h$ of $K$ must be 1 if this kind of elliptic curves exists. Rohrlich $\cite{Roh84a}$ has proved the same theorem for this kind of Hecke characters given by elliptic curves, and his calculation relies on the assumption that $K$ has class number 1. Our main result does not need the class number 1 assumption. So it can be viewed as a generalization from elliptic curves to this special kind of abelian varieties. Even earlier, Greenberg has proved the case for elliptic cuve $E$ where $P$ contains a single prime of ordinary reduction for E (see $\cite{Greenberg}$, Theorem 3 and  Proposition 8).
	
	The proof of Theorem 1 follows the pattern of $\cite{Roh84a}$. A key ingredient of $\cite{Roh84a}$ is the following conclusion. Given $\chi\in X$ and a field automorphism $\sigma$ of $\mathbb{C}$, let $\chi^\sigma:\mathfrak{a}\mapsto\chi(\mathfrak{a})^{\sigma}$. Then we have
	$$
	L(1,\chi)=0\Rightarrow L(1,\chi^{\sigma})=0,
	$$
	$$
	W(\chi)=-1 \thickspace\mbox{and}\thickspace L'(s,\chi)=0\Rightarrow L'(s,\chi^\sigma)=0.
	$$
	The first statement follows from $\cite{Shi76}, \cite{shimura77}$, and the second one follows from $\cite{GroZag83}$. These results allow us to take a suitable kind of average. Another key ingredient of $\cite{Roh84a}$ is Ridout's $p$-adic version of Roth's theorem$\cite{Ridout}$. Rohrlich uses it to give an important estimation, which we shall recall later. 
	
	Our novel points are in the transition to Roth's theorem. The main difficulty for the generalization is that the Hecke character has no good description on non-principal ideals. A novel point of our work is to find a suitable element of idele group to reflect the property of each non-principal ideal. More precisely, for each integral ideal $\mathfrak{a}$ which contributes to the average, every generator of $\mathfrak{a}^h$ must satisfy a strict restriction given by $\chi$. Another point is to give a more general version of the Main Lemma of Rohrlich -- it is necessary in our discussion to work with certain $h$-th radicals in each local part to cancel the bad influence of $h$-th powers.

	
	{\bf Acknowlegements.} This paper will be the crucial part of the author's master degree thesis under supervision of Professor Xin Wan. The author extends his deepest gratitude to Professor Xin Wan for suggesting this interesting project, for his patient and insightful guidance, and for his encouragement throughout the whole process. He is also indebted to Professor David E. Rohrlich for his paper $\cite{Roh84a}$ which inspires this work, and for his positive perspective on this problem and helpful communications passed on to the author by Professor Xin Wan. Finally, He is thankful for Haidong Li, Ruichen Xu, and Dr. Luochen Zhao for beneficial conversations. 
	
	\section{$L$-functions}
	We retain the notations as before. Define $v=v(\chi)=0$ or 1 through $W(\chi)=(-1)^{v(\chi)}$. We want to prove for all but finitely many $\chi\in X$,
	$$
	L^{(v)}(1,\chi)\neq 0.
	$$
	
	We begin with a property of equivariant Hecke characters with infinite type (1,0). Let $\epsilon_1:(\mathcal{O}/\mathfrak{f}(\phi))^{\times}\rightarrow\mathbb{C}^{\times}$ such that $\phi(w\mathcal{O})=\epsilon_1(w)w,\forall w\in K$. Here we denote $\epsilon_1(w)=0$ if $w$ is not coprime to the conductor. In particular, let $\phi(n\mathcal{O})=\kappa_1(n)n,\forall n\in \mathbb{Z}$, then $\kappa_1(n)=\phi(n\mathcal{O})n^{-1}$ is a Dirichlet character. Let $\kappa$ be the quadratic Dirichlet character corresponding to $K$. 
	
	\begin{property}
		Let $\phi$ as above, then the followings are equivalent:
		
		{\rm (1)} $\phi(\bar{\mathfrak{a}})=\bar{\phi}(\mathfrak{a})$ for every integral ideal $\mathfrak{a}$;
		
		{\rm (2)} $\kappa_1$ is a quadratic Dirichlet character, and it coincides with $\kappa$ on rational integers coprime to $\mathfrak{f}(\phi)$.
	\end{property}
	
	\begin{proof}
		(1)$\Rightarrow$(2).
		$$
		\overline{\kappa_1(n)}=\overline{\phi(n\mathcal{O})n^{-1}}=\phi(n\mathcal{O})n^{-1}=\kappa_1(n).
		$$
		This means $\kappa_1$ is a quadratic Dirichlet character. We assume $\kappa_1$ to be a character of $(\mathbb{Z}/m\mathbb{Z})^{\times}$.
		$$
		\kappa_1(-n)=\phi_1(n\mathcal{O})(-n)^{-1}=-\kappa_1(n)\Rightarrow \kappa_1(-1)=-1.
		$$
		So $\kappa_1$ is nontrivial, hence equals 1 on the half of residue classes of $(\mathbb{Z}/m\mathbb{Z})^{\times}$, and equals $-1$ on the other half.
		
		For each integral ideal $\mathfrak{a}$ of $K$ which is coprime to $\mathfrak{f}(\phi)$,
		$$
		\kappa_1(N\mathfrak{a})=\phi(N\mathfrak{a}\cdot \mathcal{O})N\mathfrak{a}^{-1}=\phi(\mathfrak{a})\phi(\bar{\mathfrak{a}})N\mathfrak{a}^{-1}=\phi(\mathfrak{a})\bar{\phi}(\mathfrak{a})N\mathfrak{a}^{-1}=1,
		$$
		where $N$ is the absolute norm, and we use the fact that: for infinite type (1,0) Hecke character $\phi$, we have $\phi(\mathfrak{a})\bar{\phi}(\mathfrak{a})=N\mathfrak{a}$. Therefore, $\kappa_1(p)=1$ for each prime $p$ which is  coprime to $\mathfrak{f}(\phi)$ and splits in $K$. These primes take up the half of residue classes of $(\mathbb{Z}/m\mathbb{Z})^{\times}$, because they have Dirichlet density 
		$\frac{1}{2}$. So the other half of residue classes of $(\mathbb{Z}/m\mathbb{Z})^{\times}$ will be taken up by primes which are inert in $K$, and $\kappa_1$ takes value $-1$ on this half. This is a kind of description for $\kappa$, so we get the conclusion.
		
		(2)$\Rightarrow$(1). If $\phi(\mathfrak{a})=\phi(\bar{\mathfrak{a}})=0$, then (1) is obvious. Hence we may assume $\phi(\mathfrak{a})\neq 0$.
		$$
		\phi \mbox{ has infinite type (1,0)} \Rightarrow\phi(\mathfrak{a})\bar{\phi}(\mathfrak{a})=N\mathfrak{a}.$$
		$$
		\kappa_1(N\mathfrak{a})=\kappa(N\mathfrak{a})=1\Rightarrow \phi(\mathfrak{a})\phi(\bar{\mathfrak{a}})=N\mathfrak{a}\Rightarrow \phi(\bar{\mathfrak{a}})=\bar{\phi}(\mathfrak{a}).
		$$
	\end{proof}
	
	We also need a technical reduction. Let $R(\chi)$ be the set of rational prime factors of $N\mathfrak{f}(\chi)$. Then $R(\chi)$ is a finite subset of 
	$R(\phi)\cup P$. In particular, there are only finitely many possibilities for $R(\chi)$. Therefore, we fix a subset $R\subseteq R(\phi)\cup P$, and let 
	$$
	Y=\{\chi\in X\ |\ R(\chi)=R\}.
	$$
	It will be sufficient to show that $L^{(v)}(1,\chi)\neq 0$ for all but finitely many $\chi\in Y$.
	
	Fix $\chi\in Y$, let $\epsilon:(\mathcal{O}/\mathfrak{f}(\chi))^{\times}\rightarrow \mathbb{C}^{\times}$, such that $\chi(w\mathcal{O})=\epsilon(w)w$ for each $w\in K^{\times}$ coprime to $\mathfrak{f}(\chi)$. We have shown that $\chi(n\mathcal{O})=\kappa(n)n^{-1},\forall n\in\mathbb{Z}$ such that $n$ is coprime to
	$\mathfrak{f}(\chi)$.
	
	Now we begin to deal with $L$-functions. Let
	\begin{equation}
		A=(2\pi)^{-1}|\mbox{discriminant of $K$}|^{\frac{1}{2}},
	\end{equation}
	\begin{equation}
		f=f(\chi)=(N\mathfrak{f}(\chi))^{\frac{1}{2}},
	\end{equation}
	\begin{equation}
		\Lambda(s,\chi)=\Gamma(s)(Af)^sL(s,\chi).
	\end{equation}
	The functional equation is	
	\begin{equation}
		\Lambda(s,\chi)=W(\chi)\Lambda(2-s,\chi).
	\end{equation}
	
	Hecke proved that
	\begin{equation}
		\Lambda(s,\chi)=\int_{1}^{\infty}\theta(t,\chi)(t^s+W(\chi)t^{2-s})\frac{dt}{t},
	\end{equation}
	where
	$$
	\theta(t,\chi)=\sum_{\mathfrak{a}}\chi(\mathfrak{a})e^{-\frac{N\mathfrak{a}}{Af}t},
	$$
	$\mathfrak{a}$ runs through integral ideals of $K$ (see $\cite{Roh84a}$, (6) and (7)). Letting $s=1$ if $v=0$, or differentiating with respect to $s$ and then letting $s=1$ if $v=1$, we have:
	\begin{equation}
		\Lambda^{(v)}(1,\chi)=2\int_{1}^{\infty}\theta(t,\chi)(\log t)^vdt.
	\end{equation}
	For $u> 0$, we have
	$$
	\int_{1}^{\infty}e^{-tu}(\log t)^vdt=\frac{I_v(u)}{u},\qquad I_v(u)=\int_{u}^{\infty}e^{-t}\Bigl(\log\frac{t}{u}\Bigr)^vdt.
	$$
	Exchanging the order of sum and integral to get
	\begin{equation}
		\Lambda^{(v)}(1,\chi)=2Af\sum_{\mathfrak{a}}\chi(\mathfrak{a})N\mathfrak{a}^{-1}I_v(\frac{N\mathfrak{a}}{Af}),
	\end{equation}
	we have
	\begin{equation}\label{16}
		L^{(v)}(1,\chi)=2\sum_{\mathfrak{a}}\chi(\mathfrak{a})N\mathfrak{a}^{-1}I_v(\frac{N\mathfrak{a}}{Af}).
	\end{equation}
	
	We divide the sum on the right side of (\ref{16}) into two parts,
	\begin{equation}
		\sum_{\mathfrak{a}}=\sum_{\mathfrak{a}=\bar{\mathfrak{a}}}+\sum_{\mathfrak{a}\neq\bar{\mathfrak{a}}}.
	\end{equation}
	If $\mathfrak{a}=\bar{\mathfrak{a}},\chi(\mathfrak{a})\neq 0$, then $\mathfrak{a}=n\mathcal{O}$ for some positive integer $n$. This is because $\kappa$ takes value 0 on primes which ramify in $K$. Therefore,
	\begin{equation}
		\chi(\mathfrak{a})=\chi(n\mathcal{O})=\kappa(n)n,
	\end{equation}
	so
	\begin{equation}
		L^{(v)}(1,\chi)=2\sum_{n\geq 1}\kappa(n)n^{-1}I_v(\frac{n^2}{Af})+2\sum_{\mathfrak{a}\neq\bar{\mathfrak{a}}}\chi(\mathfrak{a})N\mathfrak{a}^{-1}I_v(\frac{N\mathfrak{a}}{Af}).
	\end{equation}
	
	It seems that the first term can be handled easily, but it is hard to deal with the second term directly. We should take some kind of average. Let $\sigma$ be a $K$-automorphism of $\mathbb{C}$. It is easy to see that $\mathfrak{f}(\chi)=\mathfrak{f}(\chi^{\sigma})$. Although $\chi^{\sigma}$ may not be in $X$, it is still a Hecke character of $K$ with infinite type (1,0) satisfying   
	\begin{equation}
		\chi^{\sigma}(n\mathcal{O})=\kappa(n)n,
	\end{equation}
	which means $\chi^{\sigma}$ is equivariant. So we can replace $\chi$ by $\chi^\sigma$ and repeat the discussion above to $\chi^\sigma$. To take the suitable average, we need  
	\begin{equation}
		W(\chi^\sigma)=W(\chi).
	\end{equation}
	
	There is an explicit formula for $W(\chi)$ (see $\cite{Miyake}$, page 93, Theorem 3.3.1).
	Let $\mathcal{D}$ be the different ideal of $K$, i.e. the integral ideal of $K$ such that
	\begin{equation}
		\mathcal{D}^{-1}=\{a\in K|\ {\rm Tr}_{K/\mathbb{Q}}(ab)\in\mathbb{Z},\ \forall b\in \mathcal{O}\}.
	\end{equation}
	For quadratic field $K=\mathbb{Q}(\sqrt{D})$, we have
	\begin{equation}
		\mathcal{D}=\left\{
		\begin{aligned}
			(2\sqrt{D})\mathcal{O},\qquad&\mbox{if}\, D\equiv 2,3\mod 4,\\
			\sqrt{D}\mathcal{O},\qquad&\mbox{if}\,D\equiv 1 \mod 4.
		\end{aligned}\right.
	\end{equation}
	Take an integral ideal $\mathfrak{c}$ such that $\mathfrak{f}(\chi)\mathfrak{c}=b\mathcal{O}$, $b\in \mathcal{O}$, and $\mathfrak{f}(\chi)$ is coprime to $\mathfrak{c}$. Let $\delta$ be a generator of $\mathcal{D}$, then by Theorem 3.3.1 in page 93 of $\cite{Miyake}$, 
	\begin{equation}
		W(\chi)=(-i)\cdot f^{-1}\cdot\frac{\delta}{|\delta|}\cdot\frac{b}{|b|}\cdot\frac{N(\mathfrak{c})^{\frac{1}{2}}}{\chi(\mathfrak{c})}\cdot\sum_{w}\epsilon(w)e^{2\pi i {\rm Tr}_{K/\mathbb{Q}}(\frac{w}{\delta b})},
	\end{equation}
	where $w$ runs through the representatives of $\mathfrak{c}/\mathfrak{f}(\chi)\mathfrak{c}$.
	We may assume $\frac{\delta}{|\delta|}=i$. Then we substitute $f=N(\mathfrak{f(\chi)})^{\frac{1}{2}},|b|=N(\mathfrak{f(\chi)\mathfrak{c}})^{\frac{1}{2}}$, and let the coefficient $c(\chi)=\frac{b}{N(\mathfrak{f(\chi)})}\chi(\mathfrak{c})^{-1}$.
	
	Let $\xi$ be a root of unity, such that $\mathbb{Q}(\xi)$ includes $K$ and all roots of unity $e^{2\pi i {\rm Tr}_{K/\mathbb{Q}}(\frac{w}{\delta b})}$. Let $\sigma(\xi)=\xi^{m}$, where $m$ is a positive integer. We apply $\sigma$ on $W(\chi)$ to get 
	\begin{equation}
		W(\chi)=c(\chi^{\sigma})\sum_{w}\epsilon(w)^{\sigma}e^{2\pi i {\rm Tr}_{K/\mathbb{Q}}(\frac{mw}{\delta b})}.
	\end{equation}
	Taking $n\in \mathbb{Z}$ such that $nm\equiv 1 \mod \mathfrak{f(\chi)}$ and substituting $w$ by $nw$ in the above equation, we have
	\begin{equation}
		W(\chi)=\kappa(n) W(\chi^\sigma),
	\end{equation}
	where we use $\epsilon(nw)=\kappa(n)\epsilon(w)$. Note that $\kappa(n)=\kappa(m)=1$, because a description of $\kappa$ is 
	Gal$(\mathbb{Q}(\xi')/\mathbb{Q})\rightarrow {\rm Gal}(K/\mathbb{Q})\simeq \{\pm 1\}$, for a suitable root of unity $\xi'$, and we assume $\sigma$ fix $K$.
	
	Now we can deal with $\chi^\sigma$ same as $\chi$ and get
	\begin{equation}
		L^{(v)}(1,\chi^\sigma)=2\sum_{n\geq 1}\kappa(n)n^{-1}I_v(\frac{n^2}{Af})+2\sum_{\mathfrak{a}\neq\bar{\mathfrak{a}}}\chi^\sigma(\mathfrak{a})N\mathfrak{a}^{-1}I_v(\frac{N\mathfrak{a}}{Af}).
	\end{equation}
	Let $K(\chi)$ be the finite extension of $K$ generated by values of $\chi$. We consider
	\begin{equation}
		L^{(v)}(1,\chi)_{av}=[K(\chi):K]^{-1}\sum_{\sigma}L^{(v)}(1,\chi^\sigma),
	\end{equation}
	where $\sigma$ runs through a set of automorphisms of $\mathbb{C}$ which restrict to the distinct embedings of $K(\chi)$ over $K$. Then
	\begin{equation}\label{4}
		L^{(v)}(1,\chi)_{av}=2\sum_{n\geq 1}\kappa(n)n^{-1}I_v(\frac{n^2}{Af})+2\sum_{\mathfrak{a}\neq\bar{\mathfrak{a}}}\chi_{av}(\mathfrak{a})N\mathfrak{a}^{-1}I_v(\frac{N\mathfrak{a}}{Af}),
	\end{equation}
	where $\chi_{av}$ is defined as
	\begin{equation}
		\chi_{av}(\mathfrak{a})=[K(\chi):K]^{-1}\sum_{\sigma}\chi(\mathfrak{a})^\sigma=[K(\chi):K]^{-1}{\rm Tr}_{K(\chi)/K}(\chi(\mathfrak{a})).
	\end{equation}
	The reason that this average can simplify the problem is the following. Suppose $L^{(v)}(1,\chi)=0$. Then by 
	$$
	L(1,\chi)=0\Rightarrow L(1,\chi^{\sigma})=0,
	$$
	$$
	W(\chi)=-1\thickspace \mbox{and} \thickspace L'(s,\chi)=0\Rightarrow L'(s,\chi^\sigma)=0,
	$$
	which we have mentioned in the introduction, we can know that every $L^{(v)}(1,\chi^\sigma)=0$. In particular, $L^{(v)}(1,\chi)_{av}=0$. On the other hand, we shall prove that $L^{(v)}(1,\chi)_{av}\neq 0$ when $f(\chi)$ is sufficiently large. This will prove the theorem, because there are only finitely many $\chi\in Y$ such that $f(\chi)$ lies below a given bound.
	
	Now we should consider $v=0$ and $v=1$ cases separately.
	
	For $v=0$ case, 
	\begin{equation}
		I_{0}(u)=\int_{u}^{\infty}e^{-t}dt=e^{-u},
	\end{equation}
	\begin{equation}\label{11}
		L^{(v)}(1,\chi)_{av}=2\sum_{n\geq 1}\kappa(n)n^{-1}e^{-\frac{n^2}{Af}}+2\sum_{\mathfrak{a}\neq\bar{\mathfrak{a}}}\chi_{av}(\mathfrak{a})N\mathfrak{a}^{-1}e^{-\frac{N\mathfrak{a}}{Af}}.
	\end{equation}	
	According to Abel's theorem on power series, 
	\begin{equation}
		\sum_{n\geq 1}\kappa(n)n^{-1}e^{-\frac{n^2}{Af}}\rightarrow \sum_{n\geq 1}\kappa(n)n^{-1}=L(1,\kappa),
	\end{equation}
	when $f$ goes to infinity. Since $L(1,\kappa)\neq 0$, it will be sufficient to show that the second term of ($\ref{11}$) converges to 0 as $f$ goes to infinity.
	
	Fix $\chi\in Y$, let $\mathcal{N}(\chi,t)$ denote the set of integral ideals $\mathfrak{a}$ of $K$ satisfying the following conditions:
	
	(1) $\chi_{av}(\mathfrak{a})\neq 0$;
	
	(2) $\mathfrak{a}\neq \bar{\mathfrak{a}}$;
	
	(3) $N\mathfrak{a}\leq t$.
	
	\noindent Let $N(\chi,t)$ be the cardinality of $\mathcal{N}(\chi,t)$.
	
	To estimate the second term, we divide the sum into three parts
	\begin{equation}
		\sum_{\mathfrak{a}\neq\bar{\mathfrak{a}}}\chi_{av}(\mathfrak{a})N\mathfrak{a}^{-1}e^{-\frac{N\mathfrak{a}}{Af}}=\sum_{\mathfrak{a}\neq\bar{\mathfrak{a}},\,N\mathfrak{a}\leq x}+\sum_{\mathfrak{a}\neq\bar{\mathfrak{a}},\,x< N\mathfrak{a}\leq y}+\sum_{\mathfrak{a}\neq\bar{\mathfrak{a}},\,y<N\mathfrak{a}}.
	\end{equation}
	Since $|\chi_{av}(\mathfrak{a})|\leq(N\mathfrak{a})^\frac{1}{2}$, we obtain the upper bounds
	\begin{equation}\label{1}
		\Bigl|\sum_{\mathfrak{a}\neq\bar{\mathfrak{a}},\,N\mathfrak{a}\leq x}\Bigr|\leq N(\chi,x),
	\end{equation}
	\begin{equation}\label{2}
		\Bigl|\sum_{\mathfrak{a}\neq\bar{\mathfrak{a}},\,x< N\mathfrak{a}\leq y}\Bigr|\leq N(\chi,y)x^{-\frac{1}{2}},
	\end{equation}
	\begin{equation}\label{3}
		\Bigl|\sum_{\mathfrak{a}\neq\bar{\mathfrak{a}},\,y<N\mathfrak{a}}\Bigr|\leq\sum_{y<N\mathfrak{a}}N\mathfrak{a}^{-\frac{1}{2}}e^{-\frac{N\mathfrak{a}}{Af}}.
	\end{equation}
	
	We hope to take appropriate $x,y,$ such that these bounds go to 0 when $f$ goes to infinty. We shall prove the following key propositions later.
	
	\begin{proposition}
		Fix $a<1$, if $f=f(\chi)$ is sufficiently large, then
		$$
		N(\chi,f^a)=0.
		$$
	\end{proposition}
	\begin{proposition}
		There exist numbers $b>1,r<\frac{1}{2}$, such that if $f=f(\chi)$ is sufficiently large, then
		$$
		N(\chi,f^b)<f^r.
		$$
	\end{proposition}
	
	Let $b>1,r<\frac{1}{2}$ be the constants in Proposition 2, choose $a$ such that $2r<a<1$. Let $x=f^a,y=f^b$. Applying Proposition 1 and ($\ref{1}$), we obtain 
	$$
	\sum_{\mathfrak{a}\neq\bar{\mathfrak{a}},\,N\mathfrak{a}\leq x}=0,
	$$
	when $f$ is sufficiently large.
	
	Using (\ref{2}) and Proposition 2, we see that
	$$
	\Bigl|\sum_{\mathfrak{a}\neq\bar{\mathfrak{a}},\,x< N\mathfrak{a}\leq y}\Bigr|\leq f^{r-\frac{a}{2}},
	$$
	when $f$ is sufficiently large. This term goes to 0 when $f$ goes to infinity since $r<\frac{a}{2}$. Finally, for ($\ref{3}$), we have
	$$
	\Bigl|\sum_{\mathfrak{a}\neq\bar{\mathfrak{a}},\,y<N\mathfrak{a}}\Bigr|\leq\sum_{f^b<n}c(n)n^{-\frac{1}{2}}e^{-\frac{n}{Af}},
	$$
	where $c(n)$ is the number of integral ideals of $K$ of norm $n$. It is well known that $c(n)=O(n^\epsilon),$ for every $\epsilon>0,$ so
	$$
	\Bigl|\sum_{\mathfrak{a}\neq\bar{\mathfrak{a}},\,y<N\mathfrak{a}}\Bigr|\leq C\sum_{f^b<n}e^{-\frac{n}{Af}},
	$$
	where C is a constant. Because
	$$
	\sum_{f^b<n}e^{-\frac{n}{Af}}=O(fe^{-\frac{1}{A}f^{b-1}})
	$$
	and $b>1$, we have ($\ref{3}$) also goes to 0 when $f$ goes to infinity.
	
	For $v=1$ case,
	\begin{equation}\label{9}
		I_1(u)=\int_{u}^{\infty}e^{-t}\log \frac{t}{u}dt=e^{-u}\log\frac{1}{u}+J(u),
	\end{equation}
	where
	\begin{equation}\label{5}
		J(u)=\int_{u}^{\infty}e^{-t}\log tdt.
	\end{equation}
	Therefore, the first term of right side of ($\ref{4}$) has the form
	\begin{align}
		\sum_{n\geq 1}\kappa(n)n^{-1}I_1(\frac{n^2}{Af})=&\log(Af)\sum_{n\geq 1}\kappa(n)n^{-1}e^{-\frac{n^2}{Af}}\notag\\
		&-2\sum_{n\geq 1}\kappa(n)(\log n)n^{-1}e^{-\frac{n^2}{Af}}\notag\\
		&+\sum_{n\geq 1}\kappa(n)n^{-1}J(\frac{n^2}{Af}).
	\end{align}
	We consider the three terms separately.
	
	We have pointed out the sum in the first term converges to $L(1,\kappa)\neq 0,$ so the first term is asymptotic to $2\log(Af)L(1,\kappa)$. For the second term, we apply Abel's theorem as well and know that
	\begin{equation}
		\sum_{n\geq 1}\kappa(n)(\log n)n^{-1}e^{-\frac{n^2}{Af}}
	\end{equation}
	converges to $L'(1,\kappa)$, when $f$ goes to infinity. In particular, this term is bounded. Finally, using summation by parts, we obtain
	\begin{equation}\label{6}
		\sum_{n\geq 1}\kappa(n)n^{-1}J(\frac{n^2}{Af})=\sum_{n\geq 1}S(n)\Bigl(J(\frac{n^2}{Af})-J(\frac{(n+1)^2}{Af})\Bigr),
	\end{equation}
	where
	\begin{equation}
		S(n)=\sum_{1\leq m\leq n}\kappa(m)m^{-1}.
	\end{equation}
	Here we have used two facts: the first one is that $S(n)$ is bounded since the corresponding series converges; the second one is that $J(\frac{n^2}{Af})$ goes to 0 for the fixed $f$ as $n$ goes to infinity because of ($\refeq{5}$). Also note that
	\begin{equation}
		\sum_{n\geq 1}\Bigl|J(\frac{n^2}{Af})-J(\frac{(n+1)^2}{Af})\Bigr|\leq \int_{0}^{\infty}e^{-t}|\log t|dt<\infty.
	\end{equation}
	So ($\refeq{6}$) is bounded. Combining the observations above, we conclude that the first term of the right side of ($\ref{4}$) is unbounded.
	
	We begin to handle the second term of right side of ($\ref{4}$). Firstly, note that there are two estimations
	\begin{equation}\label{7}
		0<I_1(u)\leq e^{-u}|\log u|+C_0\qquad(u>0),
	\end{equation}
	\begin{equation}\label{8}
		0<I_1(u)\leq e^{-u}\qquad(u\geq 1),
	\end{equation}
	where $C_0$ is a constant. Thus ($\ref{7}$) can be deduced by ($\ref{9}$). Using integration by parts on $J(u)$, we obtain
	$$
	I_1(u)=\int_{u}^{\infty}e^{-t}\frac{dt}{t}\leq\int_{u}^{\infty}e^{-t}dt=e^{-u},
	$$
	i.e. ($\ref{8}$).
	
	Now take $a,b,r$ as before, let $x=f^a,y=f^b,$ break
	\begin{equation}
		\sum_{\mathfrak{a}\neq\bar{\mathfrak{a}}}\chi_{av}(\mathfrak{a})N\mathfrak{a}^{-1}I_1(\frac{N\mathfrak{a}}{Af})
	\end{equation}
	into
	$$
	\sum_{\mathfrak{a}\neq\bar{\mathfrak{a}}}=\sum_{\mathfrak{a}\neq\bar{\mathfrak{a}},\,N\mathfrak{a}\leq x}+\sum_{\mathfrak{a}\neq\bar{\mathfrak{a}},\,x< N\mathfrak{a}\leq y}+\sum_{\mathfrak{a}\neq\bar{\mathfrak{a}},\,y<N\mathfrak{a}}.
	$$
	The first term is 0 when $f$ is sufficiently large by Proposition 1. 
	We have estimation for the second term
	\begin{equation}
		\Bigl|\sum_{\mathfrak{a}\neq\bar{\mathfrak{a}},\,x< N\mathfrak{a}\leq y}\Bigr|\leq N(\chi,y)x^{-\frac{1}{2}}\max_{\frac{f^{a-1}}{A}<u\leq\frac{f^{b-1}}{A}}I_1(u).
	\end{equation}
	Along with ($\ref{7}$), we obtain
	\begin{equation}
		\max_{\frac{f^{a-1}}{A}<u\leq\frac{f^{b-1}}{A}}I_1(u)\leq C_1\log f+C_2,
	\end{equation}
	where $C_1,C_2$ are positive constants. Combining with Proposition 2, we get
	\begin{equation}
		\Bigl|\sum_{\mathfrak{a}\neq\bar{\mathfrak{a}},\,x< N\mathfrak{a}\leq y}\Bigr|\leq f^{r-\frac{a}{2}}(C_1\log f+C_2),
	\end{equation}
	so this term goes to 0. Finally, we deal with the third term. Note that when $f$ is sufficiently large and $N\mathfrak{a}>y$,
	$$
	\frac{N\mathfrak{a}}{Af}>\frac{f^{b-1}}{A}>1.
	$$
	So $(\ref{8})$ can be applied. Therefore,
	\begin{equation}
		\Bigl|\sum_{\mathfrak{a}\neq\bar{\mathfrak{a}},\,y<N\mathfrak{a}}\Bigr|\leq\sum_{\mathfrak{a}\neq\bar{\mathfrak{a}},\,y<N\mathfrak{a}}N\mathfrak{a}^{-\frac{1}{2}}I_1(\frac{N\mathfrak{a}}{Af})\leq \sum_{\mathfrak{a}\neq\bar{\mathfrak{a}},\,y<N\mathfrak{a}}N\mathfrak{a}^{-\frac{1}{2}}e^{\frac{N\mathfrak{a}}{Af}}.
	\end{equation}
	We have met this estimation in $v=0$ case, and it goes to 0. So we complete the proof of $v=1$ case.
	
	Actually, we have proven the following asymptotic property of $L(1,\chi)_{av}$ when $f$ goes to infinity:
	$$
	\left\{	
	\begin{aligned}
		&L(1,\chi)_{av}\rightarrow2L(1,\kappa), \qquad\qquad\qquad\mbox{if}\ W(\chi)=1,\\
		&L'(1,\chi)_{av}\sim 2L(1,\kappa)\log(Af), \qquad\ \,\mbox{if}\  W(\chi)=-1.
	\end{aligned}\right.
	$$
	
	It remains to show Propositions 1 and 2.
	
	\section{Transition to Roth's theorem}
	We shall convert Propositions 1 and 2 into propositions about pairs of rational integers. 
	We begin with a lemma about the trace of a root of unity. 
	
	\begin{lemma}
		Let $F$ be a number field, $\xi$ be a root of unity of order N. Then there is a positive integer $\mu$ only depending on $F$, such that if there exists a prime $p$ satisfying $p^\mu|N,$ then ${\rm Tr}_{F(\xi)/F}(\xi)=0$.
	\end{lemma}
	\begin{proof}
		Firstly we prove that for the fixed prime $p$, there is a positive integer $\mu_p$ depending on $F$ and $p$, such that if $p^{\mu_{p}}|N$, then ${\rm Tr}_{F(\xi)/F}(\xi)$ $=0$.
		
		When $p\nmid[F:\mathbb{Q}]$, we show that we can take $\mu_{p}=2$. Suppose $p^2|N$, then $[\mathbb{Q}(\xi):\mathbb{Q}(\xi^p)]=p$. Since $F(\xi^p)/F$ is a Galois extension, we have $[F(\xi^p):\mathbb{Q}(\xi^p)]|[F:\mathbb{Q}]$, thus $[F(\xi^p):\mathbb{Q}(\xi^p)]$ is coprime to $p$. From these two points, we obtain
		\begin{align}
			p|[\mathbb{Q}(\xi):\mathbb{Q}(\xi^p)]\cdot[F(\xi):\mathbb{Q}(\xi)]&=[F(\xi):\mathbb{Q}(\xi^p)]\notag\\
			&=[F(\xi):F(\xi^p)]\cdot[F(\xi^p):\mathbb{Q}(\xi^p)],
		\end{align}
		so
		\begin{equation}\label{12}
			p|[F(\xi):F(\xi^p)].
		\end{equation}
		$x^p-\xi^p$ is an annihilating polynomial of $\xi$, so we see that $x^p-\xi^p$ is the minimal polynomial of $\xi$ over $F(\xi^p)$ by ($\ref{12}$), hence ${\rm Tr}_{F(\xi)/F}(\xi)=0$. 
		
		When $p|[F:\mathbb{Q}]$, let $\xi=\zeta\eta$, where $\zeta$ is a root of unity with order $p^k$, and $\eta$ is a root of unity with order coprime to $p$. Let $F^p,\mathbb{Q}^p$ respectively denote the field generated by all roots of unity of $p$-power order over $F,\mathbb{Q}$. View Gal$(F^p/F)$ as a subgroup of 
		Gal$(\mathbb{Q}^p/\mathbb{Q})\simeq\mathbb{Z}_{p}^{\times}$ through Gal$(F^p/F)\simeq {\rm Gal}(\mathbb{Q}^p/\mathbb{Q}\cap F)$. Take $\mu_{p}\geq 2$ such that the inertia subgroup at some prime of $F$ over $p$ contains $1+p^{\mu_{p}-1}\mathbb{Z}_p$. Then the ramification index of $F(\zeta)/F(\zeta^p)$ over this prime is $p$ when $k\geq \mu_{p}$. Since $F(\xi)/F(\zeta),F(\xi^p)/F(\zeta^p)$ are unramified at every prime over $p$, so the ramification index of $F(\xi)/F(\xi^p)$ at some prime over $p$ is $p$ when $p^{\mu_{p}}|N$. In particular, $[F(\xi):F(\xi^p)]=p,$ thus ${\rm Tr}_{F(\xi)/F}(\xi)=0$.
		
		Finally, let
		$$
		\mu=\max_{p|[F:\mathbb{Q}]}\mu_{p},
		$$
		and we obtain the constant $\mu$ satisfying our demand. 
	\end{proof}
	
	We define some notations before explaining the purpose of Lemma 1. Fix an ideal class of $K$. Let $\mathfrak{a}_{0}$ be a fractional ideal in this class, such that $\mathfrak{a}_0$ is coprime to $\mathfrak{f}(\chi)$ and every integral ideal in this class can be written as $w\mathfrak{a}_{0}(w\in \mathcal{O})$. Since class number is finite, we may additionally assume in the definition of $\mathcal{N}(\chi,t)$ that
	
	(4) $\mathfrak{a}$ is in the ideal class of $\mathfrak{a}_0$.
	
	\noindent Let $N(\chi,t)$ be the cardinality of $\mathcal{N}(\chi,t)$ again. Then we still only need to show the description of Propositions 1 and 2.	
	
	Let $\mathfrak{a}_{0}^h=w_{0}\mathcal{O}$, fix $p\in R$, then
	$$
	w_{0}\in (\mathcal{O}\otimes\mathbb{Z}_p)^\times\simeq\prod_{v|p}\mathcal{O}_v^\times.
	$$
	Recall
	\begin{equation}
		\mathcal{O}_v^\times\simeq\mu_{q-1}\oplus\mu_{p^a}\oplus \mathbb{Z}_{p}^{[k_v:\,\mathbb{Q}_{p}]},
	\end{equation}
	where $q$ is the cardinality of the residue field, $\mu$ with subscript denotes group of roots of unity of certain order, $a$ is a positive integer. Hence we have
	\begin{equation}
		(\mathcal{O}\otimes\mathbb{Z}_{p})^\times\simeq \mbox{finite group}\oplus\mathbb{Z}_{p}^2.
	\end{equation}
	
	Let
	\begin{equation}
		\epsilon=\prod_{p\in R}\epsilon_{p},
	\end{equation}
	where $\epsilon_{p}$ is a character of $(\mathcal{O}\otimes\mathbb{Z}_{p})^\times$, extended to $\mathcal{O}\otimes\mathbb{Z}_{p}$ by requiring it to be 0 on nonunits.
	This decomposition can be found in Lecture 2 of $\cite{Rohrnotes}$, $\epsilon_{p}$ is the component of the finite type of $\chi$ over $v(v|p)$ in the idelic definition of $\chi$.
	
	Now we explain the purpose of Lemma 1. Let $\mathfrak{a}=w\mathcal{O}$. Note that	
	$$
	\chi(\mathfrak{a})^h=\chi(\mathfrak{a}^h)=\chi(w^hw_0\mathcal{O})=\epsilon(w^hw_0)w^hw_0,
	$$
	thus $\chi(\mathfrak{a})$ has the form $\zeta w\sqrt[h]{w_0}$, where $\sqrt[h]{w_0}$ is a fixed $h$-th radical of $w_0$ in $\mathbb{C}$ (we do not care which one is it specifically), $\zeta$ is a root of unit. Hence
	\begin{align}
		\chi_{av}(\mathfrak{a})&={\rm Tr}_{K(\chi)/K}(\zeta w\sqrt[h]{w_0})\notag\\
		&=\frac{1}{n}{\rm Tr}_{K(\chi,\sqrt[h]{w_0})/K}(\zeta w\sqrt[h]{w_0})\notag\\
		&=\frac{1}{n}{\rm Tr}_{K(\sqrt[h]{w_0})/K}({\rm Tr}_{K(\chi,\sqrt[h]{w_0})/K(\sqrt[h]{w_0})}(\zeta w\sqrt[h]{w_0}))\notag\\
		&=\frac{1}{n}{\rm Tr}_{K(\sqrt[h]{w_0})/K}(w\sqrt[h]{w_0}{\rm Tr}_{K(\chi,\sqrt[h]{w_0})/K(\sqrt[h]{w_0})}(\zeta)),
	\end{align}
	where $n=[K(\chi,\sqrt[h]{w_0}):K(\chi)]$.
	Therefore,
	$$
	\chi_{av}(\mathfrak{a})\neq 0\Rightarrow {\rm Tr}_{K(\chi,\sqrt[h]{w_0})/K(\sqrt[h]{w_0})}(\zeta)\neq 0.
	$$
	
	Applying Lemma 1 to $K(\sqrt[h]{w_0})$, we obtain: 
	
	\emph{There exists a positive integer $\mu_1$, such that $\chi_{av}(\mathfrak{a})\neq 0\Rightarrow$ the order of $\zeta$ is not divisible by $p^{\mu_1}$ for any prime $p$.}
	
	\noindent And thus:
	
	\emph{There exists a positive integer $\mu_2,$ such that $\chi_{av}(\mathfrak{a})\neq 0\Rightarrow$ the order of $\epsilon(w^hw_0)=\zeta^h$ is not divisible by $p^{\mu_2}$ for any prime $p$,}
	
	\noindent Therefore:
	
	\emph{There exists a positive integer $\mu,$ such that $\chi_{av}(\mathfrak{a})\neq 0\Rightarrow$ the order of $\epsilon_{p}(w^hw_0)$ is not divisible by $p^\mu$ for any $p\in R$.}
	
	This is because for every $w\in \mathcal{O}$,
	$$
	\epsilon(w)=\epsilon_{p}(w)\prod_{l\in R,l\neq p}\epsilon_{l}(w),
	$$
	and the power of $p$ in the order of $\epsilon_{l}(w)$ is controlled by a constant because $(\mathcal{O}\otimes\mathbb{Z}_l)^\times$ is the product of a finite group and a pro-$l$ group. Thus the statement that the order of $\epsilon(w)$ is divisible by $p^\mu$ is equivalent to that the order of $\epsilon_p(w)$ is divisible by $p^\mu$ when $\mu$ is sufficiently large.
	
	The assertion we shall use later is:
	
	\emph{There exists a positive integer $\mu,$ such that $\chi_{av}(\mathfrak{a})\neq 0\Rightarrow$ the order of $\epsilon_{p}(w^hw_0)$ is not divisible by $p^\mu$ for any $p\in R$.}
	
	We view $\mathbb{Z}_p^\times\simeq(\mathbb{Z}\otimes\mathbb{Z}_p)^\times$ as a subgroup of $(\mathcal{O}\otimes\mathbb{Z}_p)^\times$, and define
	$$
	H_p=\{a\in(\mathcal{O}\otimes\mathbb{Z}_p)^\times|\ a^m\in\mathbb{Z}_p^\times\ \mbox{for some positive integer}\ m\},
	$$
	then $\mathbb{Z}_p^\times$ is of finite index in $H_p$. For every $p\in P$, fix a set $\Omega_p$ of representatives of $H_p/\mathbb{Z}_p^\times$.
	Define
	$$
	S_p=\{a\in (\mathcal{O}\otimes\mathbb{Z}_p)^\times|\ a^h\in H_pw_0^{-1}\}.
	$$
	$S_p$ may be empty. When $S_p$ is not empty, $\forall x_1,x_2\in S_p$, we have $x_1^hx_2^{-h}\in H_p$, thus $x_1x_2^{-1}\in H_p$. We fix $x_p\in S_p$ if $S_p$ is not empty, then every element in $S_p$ can be written as $\eta_p\omega_p x_p,$ where $\eta_p\in\mathbb{Z}_p^\times,\omega_p\in\Omega_p$. Let 
	$$
	P'=\{p\in P|\ S_p\ \mbox{is not empty}\}.
	$$
	\begin{lemma}{\rm (Main Lemma)}.
		For every $\chi\in Y$, there exists a positive integer 
		$$
		q(\chi)=\prod_{p\in P\cap R}p^{n_p(\chi)},
		$$
		satisfying the following conditions:
		
		{\rm (1)} Assume $\chi_{av}(\mathfrak{a})\neq 0$ for some $\mathfrak{a}=w\mathfrak{a}_0$, then for every $p|q(\chi)$, we have $p\in P'$, and there exist $\eta_p\in \mathbb{Z}_p^\times,\omega_p\in \Omega_p$ such that
		$$
		w\equiv \eta_p\omega_px_p\mod q(\chi)\mathcal{O}\otimes\mathbb{Z}_p.
		$$
		
		{\rm (2)} $f(\chi)\leq k_0q(\chi),q(\chi)\leq k_1f(\chi)$, where $k_0,k_1$ are positive integers independent of $\chi$.
	\end{lemma}
	
	\begin{proof}
		Note that we may assume $\mu$ is large enough such that $p^\mu$ annihilates the Sylow $p$-subgroup of $(\mathcal{O}\otimes\mathbb{Z}_p)^\times/(1+p^3\mathcal{O}\otimes\mathbb{Z}_p)$. For every $\chi\in Y$, define $m_p(\chi),\,p\in R$ through 
		\begin{equation}
			N\mathfrak{f}(\chi)=\prod_{p\in R}p^{m_p(\chi)}.
		\end{equation}
		Let
		\begin{equation}\label{10}
			\mbox{order of}\ \epsilon_{p}|_{(1+p^3\mathcal{O}\otimes\mathbb{Z}_p)}=p^{o_p(\chi)}.
		\end{equation}
		For every $p\in P\cap R$, define
		\begin{equation}
			n_p(\chi)=\left\{
			\begin{aligned}
				&0,\qquad&\mbox{if}\ o_p(\chi)\leq\mu+h,\\
				&o_p(\chi)-\mu-h,\qquad& \mbox{if}\ o_p(\chi)>\mu+h.
			\end{aligned}\right.
		\end{equation}
		
		Now we analyse the kernel of $\epsilon_{p}|_{(1+p^3\mathcal{O}\otimes\mathbb{Z}_p)}$. The restriction of $\epsilon_{p}$ on $\mathbb{Z}_p^\times$ is $\kappa_p$, which is the $p$-component of $\kappa$, hence a quadratic character. In particular, $\epsilon_{p}$ is trivial on $1+p^3\mathbb{Z}_p$, because every element in $1+p^3\mathbb{Z}_p$ is a square. On the other hand, $\epsilon_{p}$ is also trivial on $1+p^{3+o_p(\chi)}\mathcal{O}\otimes\mathbb{Z}_p$, because we have ($\ref{10}$) and every element in $1+p^{3+o_p(\chi)}\mathcal{O}\otimes\mathbb{Z}_p$ is a $p^{o_p(\chi)}$-th power. Since
		$$
		(1+p^3\mathcal{O}\otimes\mathbb{Z}_p)/(1+p^3\mathbb{Z}_p)(1+p^{3+o_p(\chi)}\mathcal{O}\otimes\mathbb{Z}_p)
		$$
		is a cyclic group of order $p^{o_p(\chi)}$, we obtain
		\begin{equation}\label{13}
			\mbox{kernel of}\ \epsilon_{p}|_{(1+p^3\mathcal{O}\otimes\mathbb{Z}_p)}=(1+p^3\mathbb{Z}_p)(1+p^{3+o_p(\chi)}\mathcal{O}\otimes\mathbb{Z}_p).
		\end{equation}
		
		We begin to verify (1). $p|q(\chi)$ means $o_p(\chi)>\mu+h$. $\chi_{av}(\mathfrak{a})\neq 0\Rightarrow p^\mu\nmid \mbox{order of}\ \epsilon_{p}(w^hw_0)$. Since $p^\mu$ annihilates the Sylow $p$-subgroup of $(\mathcal{O}\otimes\mathbb{Z}_p)^\times/(1+p^3\mathcal{O}\otimes\mathbb{Z}_p)$ as we have assumed, we have
		\begin{equation}
			w^{mh}w_0^m\in(1+p^3\mathcal{O}\otimes\mathbb{Z}_p),
		\end{equation}
		where $m=jp^\mu,\,j$ is a positive integer coprime to $p$.
		
		The order of $\epsilon_{p}(w^hw_0)$ is not divisible by $p^\mu$, so the order of $\epsilon_{p}(w^{mh}w_0^m)$ is coprime to $p$. Then $\epsilon_{p}(w^{mh}w_0^m)$ has to be 1 as the order of $\epsilon_{p}(w^{mh}w_0^m)$ is a power of $p$. Because of (\ref{13}), we can assume		
		\begin{equation}
			zw^{mh}w_0^m\in 1+p^3\mathbb{Z}_p,\qquad z\in 1+p^{3+o_p(\chi)}\mathcal{O}\otimes\mathbb{Z}_p.
		\end{equation}
		Write $z$ as
		\begin{equation}
			z=z_1^m,\qquad z_1\in1+p^{n_p(\chi)+h}\mathcal{O}\otimes\mathbb{Z}_p.
		\end{equation}
		So
		\begin{equation}
			(z_1w^hw_0)^m\in 1+p^{3}\mathbb{Z}_p.
		\end{equation}
		In particular,
		\begin{equation}
			z_1w^hw_0=\eta_p'\omega_p',\qquad \eta_p'\in\mathbb{Z}_p^\times,\ \omega_p'\in\Omega_p.
		\end{equation}
		Write $z_1$ as
		\begin{equation}
			z_1=z_2^h,\qquad z_2\in1+p^{n_p(\chi)}\mathcal{O}\otimes\mathbb{Z}_p,
		\end{equation}
		then
		\begin{equation}
			(z_2w)^h=\eta_p'\omega_p'w_0^{-1}.
		\end{equation}
		So $S_p$ is not empty, i.e. $p\in P'$, and we can suppose
		\begin{equation}
			z_2w=\eta_p\omega_px_p,\qquad \eta_p\in\mathbb{Z}_p^\times,\ \omega_p\in\Omega_p.
		\end{equation}
		Hence we obtain the congruence equation
		\begin{equation}
			w\equiv\eta_p\omega_px_p\mod q\mathcal{O}\otimes\mathbb{Z}_p.
		\end{equation}
		
		Now we verify (2). For every $p\in P\cap R,$ let $\mathfrak{f}(\epsilon_{p})$ be the conductor of $\epsilon_{p}$, i.e. the maximal integral ideal of $k$ such that 
		\begin{equation}
			\epsilon_{p}|_{(1+\mathfrak{f}(\epsilon_{p})\otimes\mathbb{Z}_p)}=1,
		\end{equation}
		then $N\mathfrak{f}(\epsilon_{p})=p^{m_p(\chi)}$.
		
		If $o_p(\chi)\leq \mu+h,$ then $n_p(\chi)=0,\ \epsilon_{p}$ is trivial on $1+p^{3+\mu+h}\mathcal{O}\otimes\mathbb{Z}_p$. Hence
		\begin{equation}\label{14}
			\Bigl|\frac{m_p(\chi)}{2}-n_p(\chi)\Bigr|\leq3+\mu+h.
		\end{equation}
		
		If $o_p(\chi)> \mu+h$, then $n_p(\chi)=o_p(\chi)-\mu-h$, $\epsilon_p$ is trivial on $1+p^{3+o_p(\chi)}\mathcal{O}\otimes\mathbb{Z}_p$, but is nontrivial on $1+p^{2+o_p(\chi)}\mathcal{O}\otimes\mathbb{Z}_p$. So we also have
		$$
		\Bigl|\frac{m_p(\chi)}{2}-n_p(\chi)\Bigr|\leq3+\mu+h.
		$$
		
		Let
		\begin{equation}
			k=\prod_{p\in R,\,p\notin P}p^{\frac{m_p(\chi)}{2}},
		\end{equation}
		\begin{equation}
			k'=\prod_{p\in P\cap R}p^{3+\mu+h}.
		\end{equation}
		Let $k_0=kk',k_1=k'/k$. Note that  
		\begin{equation}
			f(\chi)=\prod_{p\in R,\,p\notin P}p^{\frac{m_p(\chi)}{2}}\cdot\prod_{p\in P\cap R}p^{\frac{m_p(\chi)}{2}},
		\end{equation}
		\begin{equation}
			q(\chi)=\prod_{p\in P\cap R}p^{n_p(\chi)}.
		\end{equation}
		Along with ($\ref{14}$), we obtain (2).
	\end{proof}
	
	If every integral ideal $\mathfrak{a}$ in the ideal class of $\mathfrak{a}_0$ satifies $\chi_{av}(\mathfrak{a})=0,$ then Propositions 1 and 2 hold obviously. We can always assume that there exists $\mathfrak{a}$ such that $\chi_{av}(\mathfrak{a})\neq 0,$ thus (1) of the Main lemma implies that $q(\chi)$ is only divisible by primes in $P'$. 
	
	Now we begin to apply the Main Lemma. Let 
	$$
	\Omega=\prod_{p\in P'}\Omega_p,
	$$
	which is a finite set. For $\omega\in\Omega,p\in P'$, let $\omega_p$ be the $p$-component of $\omega$. Take $\omega\in\Omega$, positive real number $t$, and a positive integer $q$ which is only divisible by primes in $P'$. Define $\mathcal{N}_\omega(q,t)$ to be the set of $w\in \mathcal{O}$ satisfying the following conditions:
	
	(1) For every $p\in P',$ there exists $\eta_p\in\mathbb{Z}_p^\times$ such that
	$$
	w\equiv\eta_p\omega_p x_p \mod q\mathcal{O}\otimes\mathbb{Z}_p;
	$$
	
	(2) $w^hw_0\mathcal{O}\neq\overline{w^hw_0}\mathcal{O};$
	
	(3) $|w|\leq t$.
	
	\noindent Let $N_\omega(q,t)$ be the cardinality of $\mathcal{N}(q,t)$.
	
	Recall that we have modified the definition of $\mathcal{N}(\chi,t)$ to be the set of integral ideals of $K$ satifying:
	
	(1) $\chi_{av}(\mathfrak{a})\neq 0;$
	
	(2) $\mathfrak{a}\neq\bar{\mathfrak{a}};$
	
	(3) $N\mathfrak{a}\leq t;$
	
	(4) $\mathfrak{a}$ is in the ideal class of $\mathfrak{a}_0$.
	
	In view of the Main Lemma, there exists a map  
	$$
	\mathcal{N}(\chi,t)\rightarrow\bigcup_{\omega\in\Omega}\mathcal{N}_{\omega}(q(\chi),lt^\frac{1}{2}).
	$$
	The map assigns to each $\mathfrak{a}=w\mathfrak{a}_0$ an arbitrarily chosen generator $w$, and the constant $l=|w_0|^{-\frac{1}{h}}$. The map is injective since these ideals are uniquely determined by their generators. So
	$$
	N(\chi,t)\leq\sum_{\omega\in\Omega}N_\omega(q(\chi),lt^\frac{1}{2}).
	$$
	Therefore it suffices to show the following form of Propositions 1 and 2. 
	\begin{proposition}{\rm (Second form of Propositions 1 and 2)}.
		Fix $c<\frac{1}{2}$, if $q$ is sufficiently large, then
		$$
		N_\omega(q,q^c)=0.
		$$
		Moreover, there exist absolute constants $d>\frac{1}{2},s<\frac{1}{2},$ such that if $q$ is sufficiently large, then
		$$
		N_\omega(q,q^d)<q^s.
		$$
	\end{proposition}
	
	Let us see how to deduce the first form of Propositions 1 and 2 from the second form. Fix $a<1,$ choose $c$ such that $\frac{a}{2}<c<\frac{1}{2}$. From (2) of the Main Lemma, when $q(\chi)$ is sufficiently large,
	\begin{equation}
		f(\chi)^{\frac{a}{2}}\leq(k_0q(\chi))^{\frac{a}{2}}\leq l^{-1}q(\chi)^c.
	\end{equation}
	Hence when $q(\chi)$ is sufficiently large,
	\begin{equation}
		N(\chi,f(\chi)^a)\leq \sum_{\omega\in\Omega}N_\omega(q(\chi),q(\chi)^c),
	\end{equation}
	the right side is 0. Thus we obtain the first form of Proposition 1.
	
	Next, take $d>\frac{1}{2},s<\frac{1}{2}$ as in the second form of Proposition 2, and choose $b,r$ such that $2d>b>1,s<r<\frac{1}{2}$. Therefore, when $f(\chi)$ is sufficiently large,
	\begin{equation}
		N(\chi,f(\chi)^b)\leq \sum_{\omega\in\Omega}(q(\chi),q(\chi)^d)
		<\sum_{\omega\in\Omega}q(\chi)^s
		<f(\chi)^r,
	\end{equation}
	i.e. the first form of Proposition 2 holds true. 
	
	Fix $\omega\in\Omega,$ replace $\mathcal{N}_\omega(q,t),N_\omega(q,t)$ by $\mathcal{N}(q,t),N(q,t)$ respectively. Let $\tau$ be an element of $\mathcal{O}$ such that $\{1,\tau\}$ is a basis for $\mathcal{O}$ over $\mathbb{Z}$, then for each $p\in P',\{1,\tau\otimes 1\}$ is a basis for $\mathcal{O}\otimes\mathbb{Z}_p$ over $\mathbb{Z}_p$. In particular, for every $p\in P'$, there exist unique $\alpha_p,\beta_p\in\mathbb{Z}_p$ such that
	\begin{equation}
		\omega_px_p=\alpha_p+\beta_p\tau\otimes 1.
	\end{equation}
	
	Let $\mathcal{M}(q,t)$ be the set of all pairs $(u,v)$ of rational integers satisfying:
	
	(1) For every $p\in P',u\beta_p-v\alpha_p\equiv 0 \mod q\mathbb{Z}_p$;
	
	(2) For every $p\in P',u\beta_p-v\alpha_p\neq 0$;
	
	(3) $|u|,|v|\leq t$.
	
	\noindent Let $M(q,t)$ be the cardinality of $\mathcal{M}(q,t)$. Our purpose is to convert Propositions 1 and 2 into a description of $M(q,t)$.
	
	Since all norms on a Euclidean space are equivlant, there exists a constant $k$ such that
	\begin{equation}\label{15}
		\max(|x|,|y|)\leq k|x+\tau y|,\qquad x,\,y\in \mathbb{R}.
	\end{equation}
	Now we verify 
	$$w=u+v\tau\mapsto(u,v)$$
	is a map from $\mathcal{N}(q,t)$ to $\mathcal{M}(q,kt)$. From
	$$w\equiv\eta_p\omega_px_p\mod q\mathcal{O}\otimes\mathbb{Z}_p,$$
	we have for each $p\in P'$,
	\begin{equation}
		u\equiv\eta_p\alpha_p\mod q\mathbb{Z}_p,
	\end{equation}
	\begin{equation}
		v\equiv\eta_p\beta_p\mod q\mathbb{Z}_p.
	\end{equation}
	Multiplying the first congruence by $\beta_p$ and the second by $\alpha_p$, and subtracting, we obtain
	\begin{equation}
		u\beta_p-v\alpha_p\equiv0\mod q\mathbb{Z}_p,
	\end{equation}
	i.e. $(u,v)$ satisfies (1) of $\mathcal{M}(q,kt)$.
	
	Suppose $u\beta_p-v\alpha_p=0$. Since $\alpha_p,\beta_p$ are not both $0$, we can write 
	\begin{equation}
		(u,v)=\eta(\alpha_p,\beta_p),\qquad \eta\in\mathbb{Q}_p,
	\end{equation}
	then $w=\eta\omega_px_p$ in $K\otimes\mathbb{Q}_p$.
	From $x_p\in S_p$, we have
	\begin{equation}
		w^h=\eta^h\omega_p^hx_p^h=\eta'\omega'w_0^{-1}\in K\otimes\mathbb{Q}_p,\qquad\eta'\in\mathbb{Q}_p,\ \omega'\in H_p.
	\end{equation}
	Take positive integer $m$ such that $(\omega')^m\in\mathbb{Z}_p^\times$. Then
	\begin{equation}
		w^{mh}w_0^m=(\eta'\omega')^m\in\mathbb{Q}_p\cap \mathcal{O}=\mathbb{Z},
	\end{equation}
	contradicting to $w^hw_0\mathcal{O}\neq\overline{w^hw_0}\mathcal{O}$. So $u\beta_p-v\alpha_p\neq 0,$ i.e. $(u,v)$ satifies (2) of $\mathcal{M}(q,kt)$.
	
	Finally, if $|w|\leq t,$ then $|u|,|v|\leq kt$ (by (\ref{15})), i.e. $(u,v)$ satifies (3) of $\mathcal{M}(q,kt)$. So $w=u+v\tau\mapsto(u,v)$ really defines a map from $\mathcal{N}(q,t)$ to $\mathcal{M}(q,kt)$, and the map is obviously injective. Therefore,
	\begin{equation}
		N(q,t)\leq M(q,kt).
	\end{equation}
	So it is sufficient to show the following form of Propositions 1 and 2. 
	
	\begin{proposition}{\rm (Third form of Propositions 1 and 2)}.
		Fix $c<\frac{1}{2},$ if $q$ is sufficiently large, then
		$$
		M(q,q^c)=0.
		$$
		Moreover, there exist absolute constants $d>\frac{1}{2},s<\frac{1}{2}$ such that if $q$ is sufficiently large, then
		$$
		M(q,q^d)<q^s.
		$$
	\end{proposition}
	
	Since $N(q,t)\leq M(q,kt)$, we can easily deduce the second form from the third form. Fix $c<\frac{1}{2}$, choose $c'$ such that $c<c'<\frac{1}{2}$. Then if $q$ is sufficiently large, we have
	\begin{equation}
		N(q,q^c)\leq M(q,kq^c)\leq M(q,q^{c'})=0,
	\end{equation}
	i.e. the second form of Proposition 1. The argument of Proposition 2 is similar.
	
	We show that we may assume $\alpha_p,\beta_p$ to be algebraic over $\mathbb{Q}$ in the following lemma.
	
	\begin{lemma}
		$(\alpha_p,\beta_p)$ can be written as $(\alpha_p,\beta_p)=\eta_p(\alpha_p',\beta_p'),$ where $\alpha_p',\beta_p'$ is an element in $\mathbb{Z}_p$ which is algebraic over $\mathbb{Q}$, $\eta_p\in\mathbb{Z}_p^\times$.
	\end{lemma}
	
	\begin{proof}
		For each nonnegative integer $j$, define $a_j,b_j\in\mathbb{Z}$ through $\tau^j=a_j\tau\otimes 1+b_j$. Take positive integer $m$ such that
		\begin{equation}
			(\omega_px_p)^{mh}\in\mathbb{Z}_p^\times w_0^{-m},
		\end{equation}
		then
		\begin{equation}
			(\omega_px_p)^{mh}w_0^{m}\in\mathbb{Z}_p^\times,
		\end{equation}
		\begin{equation}
			(\omega_px_p)^{mh}=\left(\sum_{j=0}^{mh}b_j\binom{mh}{j}\alpha_p^{mh-j}\beta_p^j\right)+\left(\sum_{j=0}^{mh}a_j\binom{mh}{j}\alpha_p^{mh-j}\beta_p^j\right)\tau\otimes 1.
		\end{equation}
		Let
		\begin{equation}
			w_0^{m}=u_0+v_0\tau\otimes 1,\qquad u_0,\ v_0\in\mathbb{Z},
		\end{equation}
		then
		\begin{align}
			(\omega_px_p)^{mh}w_0^{m}=&
			\left(\sum_{j=0}^{mh}(u_0b_j+v_0b_2a_j)\binom{mh}{j}\alpha_p^{mh-j}\beta_p^j\right)\notag\\&+\left(\sum_{j=0}^{mh}(v_0b_j+(u_0+v_0a_2)a_j)\binom{mh}{j}\alpha_p^{mh-j}\beta_p^j\right)\tau\otimes 1.
		\end{align}
		Therefore,
		\begin{equation}
			\sum_{j=0}^{mh}(v_0b_j+(u_0+v_0a_2)a_j)\binom{mh}{j}\alpha_p^{mh-j}\beta_p^j=0.
		\end{equation}
		Note that $a_0=0,b_0=1,a_1=1,b_1=0$.	
		When $v_0\neq 0$, we have $v_0b_0+(u_0+v_0a_2)a_0=v_0\neq 0$. 
		When $v_0=0$, we have $u_0\neq 0,v_0b_1+(u_0+v_0a_2)a_1=u_0\neq 0$. 
		Hence $(\alpha_p,\beta_p)$ always satisfies a nontrivial homogeneous polynomial with coefficients in $\mathbb{Z}$. This implies the desired conclusion.		
	\end{proof}
	
	Since the conditions of $\mathcal{M}(q,t)$ are homogeneous, we may assume $\alpha_p,\beta_p$ are algebraic over $\mathbb{Q}$. We see that the desired conclusions have no relation with $p\in P\backslash P'$. So we may assume $P=P'$ for simplicity of notation.
	
	The third form of Propositions 1 and 2 is already proved in $\cite{Roh84a}$ by applying Roth's theorem. We shall state his result in the next section and complete our proof.
	
	\section{Roth's theorem}
	We restate the settings of the third form of Propositions 1 and 2 in this section. This is the same as the third section of $\cite{Roh84a}$. But for convenience of the readers, we still include part of the proof here to show why it is related with Roth's theorem.
	
	Let $P$ be a finite set of rational primes. For each $p\in P$, let $\alpha_p,\beta_p$ be elements of $\mathbb{Z}_p$ which are algebraic over $\mathbb{Q}$ and are not both 0. Given a positive real number $t$, and a positive integer $q$ only divisible by primes in $P$, let $\mathcal{M}(q,t)$ be the set of all pairs of rational integers $(u,v)$ satisfying:
	
	(1) For every $p\in P$,
	
	$$
	u\beta_p-v\alpha_p\equiv0 \mod \mathbb{Z}_p;
	$$
	
	(2) For every $p\in P$,
	
	$$
	u\beta_p-v\alpha_p\neq 0;
	$$
	
	(3) $|u|,|v|\leq t$.
	
	\noindent Let $\mathcal{M}(q,t)$ denote the cardinality of $M(q,t)$.
	
	\begin{proposition}
		Fix $c<\frac{1}{2},$ if $q$ is sufficiently large, then
		$$
		M(q,q^c)=0.
		$$
		Moreover, there exist absolute constants $d>\frac{1}{2},s<\frac{1}{2}$ such that if $q$ is sufficiently large, then
		$$
		M(q,q^d)<q^s.
		$$
	\end{proposition}
	\begin{proof}
		We only give the proof for $M(q,q^c)=0$ and refer to Section 3 of $\cite{Roh84a}$ for the second statement.
		
		Suppose that for each $p\in P$ we are given a $\gamma_p\in\mathbb{Q}_p$ which is algebraic over $\mathbb{Q}$. Let $\kappa$ be a real number greater than 2. In view of Ridout's generalization of Roth's theorem, we know that the inequality
		\begin{equation}
			\prod_{p\in P}|u-v\gamma_p|_p\leq(\max(|u|,|v|))^{-\kappa}
		\end{equation}
		has finitely many solutions in coprime integers $u$ and $v$. Then there exists a positive constant $k$ such that
		\begin{equation}
			\prod_{p\in P}|u-v\gamma_p|_p>k(\max(|u|,|v|))^{-\kappa},
		\end{equation}
		as long as $u,v$ are integers such that the left side is not 0. Note that there is no need to add coprime condition, because if the inequality holds for a coprime pair $(u,v)$, then by
		\begin{equation}
			\prod_{p\in P}|n|_p\geq |n|^{-1}\geq |n|^{-\kappa},
		\end{equation}
		we know that the inequality also holds if we replace $(u,v)$ by $(nu,nv)$.
		
		Take $j$ such that $\beta_p+j\alpha_p\neq 0,\forall p\in P$. Let
		\begin{equation}
			\gamma_p=\alpha_p(\beta_p+j\alpha_p)^{-1}.
		\end{equation}
		Choose $\kappa$ such that
		\begin{equation}
			c<\frac{1}{\kappa}<\frac{1}{2}.
		\end{equation}
		If $u,v$ are integers such that $u\beta_p-v\alpha_p\neq 0,\forall p\in P$, then
		\begin{align}
			\prod_{p\in P}|u\beta_p-v\alpha_p|_p&=\prod_{p\in P}|\beta_p+j\alpha_p|_p|u-(v+ju)\gamma_p|_p\notag\\
			&>\Bigl(\prod_{p\in P}|\beta_p+j\alpha_p|_p\Bigr)k(\max(|u|,|v+ju|))^{-\kappa}.
		\end{align}
		Therefore,
		\begin{equation}
			\prod_{p\in P}|u\beta_p-v\alpha_p|_p>k_1(\max(|u|,|v+ju|))^{-\kappa},
		\end{equation}
		where
		$$
		k_1=\Bigl(\prod_{p\in P}|\beta_p+j\alpha_p|_p\Bigr) k (1+|j|)^{-\kappa}.
		$$
		
		If $(u,v)\in\mathcal{M}(q,q^c)$, we have 
		\begin{equation}
			\prod_{p\in P}|u\beta_p-v\alpha_p|_p\leq q^{-1}
		\end{equation}
		and 
		\begin{equation}
			\max(|u|,|v|)\leq q^{-1}.
		\end{equation}
		Therefore,
		\begin{equation}
			q^{-1}>k_1q^{-\kappa c}.
		\end{equation}
		Since $\kappa c<1$, we get contradiction when $q$ is sufficiently large. So if $q$ is sufficiently large, we have $M(q,q^c)=0$.
	\end{proof}
	So far, we complete the proof of Theorem 1.

	Academy of Mathematics and Systems Science, Chinese Academy of Sciences, No. 55, Zhong Guan Cun East
	Road, Beijing, 100190, China.
	
	E-mail address: jiahaijun22@mails.ucas.ac.cn

\end{document}